\documentclass[11pt]{amsart}

\usepackage{amsmath,amssymb,amsthm}
\usepackage{hyperref}
\usepackage[latin1]{inputenc}
\usepackage{color}
\usepackage{import}
\usepackage{graphicx}
\usepackage{fullpage}
\usepackage{color}
\newtheoremstyle{myremark} 
    {7pt}                    
    {7pt}                    
    {}  	                 
    {}                           
    {\bf}       	         
    {.}                          
    {.5em}                       
    {}  

\theoremstyle{plain}
\newtheorem{lemma}{Lemma}[section]
\newtheorem{theorem}[lemma]{Theorem}

\newtheorem{definition}[lemma]{Definition}

\newtheorem{proposition}[lemma]{Proposition}

\newtheorem*{main}{Main Theorem}
\newtheorem*{4.3}{Theorem~4.3}
\theoremstyle{myremark}
\newtheorem{remark}[lemma]{Remark}
\newtheorem{example}[lemma]{Example}

\newcommand{\NN}{\mathbb{N}}

\newcommand{\RR}{\mathbb{R}}
\newcommand{\QQ}{\mathbb{Q}}
\newcommand{\ZZ}{\mathbb{Z}}

\newcommand{\htpyequiv}{\simeq}

\renewcommand{\subset}{\subseteq}

\newcommand{\vrc}[2]{\mathbf{VR}(#1;#2)}
\newcommand{\wf}{\mathrm{wf}}
\newcommand{\wfc}[2]{\mathrm{wf}(#1;#2)}

\newcommand{\ddist}{\vec{d}}

\newcommand{\olen}{\ell}
\newcommand{\owind}{\mathrm{wn}}
\newcommand{\per}{\mathrm{per}}
\newcommand{\swi}{\mathrm{swi}}
\newcommand{\numorb}{\mathrm{orb}}
\newcommand{\numlev}{\mathrm{lev}}

\newcommand{\md}[1]{\ \mathrm{mod}\ #1}

\newcommand{\prob}{\mathbf{P}}
\newcommand{\expect}{\mathbf{E}}

\begin{document}
\title{Random cyclic dynamical systems}
\author{Micha{\l} Adamaszek}
\address{Department of Mathematical Sciences, University of Copenhagen, Denmark}
\email{aszek@mimuw.edu.pl}
\thanks{}
\author{Henry Adams}
\address{Department of Mathematics, Colorado State University, USA}
\email{adams@math.colostate.edu}
\author{Francis Motta}
\address{Department of Mathematics, Duke University, USA}
\email{motta@math.duke.edu}
\thanks{MA supported by VILLUM FONDEN through the network for Experimental Mathematics in Number Theory, Operator Algebras, and Topology. HA supported in part by Duke University and by the Institute for Mathematics and its Applications with funds provided by the National Science Foundation. The research of MA and HA was supported through the program ``Research in Pairs'' by the Mathematisches Forschungsinstitut Oberwolfach in 2015.}
\maketitle
\begin{abstract} 
For $X$ a finite subset of the circle and for $0<r\leq 1$ fixed, consider the function $f_r\colon X\to X$ which maps each point to the clockwise furthest element of $X$ within angular distance less than $2\pi r$. We study the discrete dynamical system on $X$ generated by $f_r$, and especially its expected behavior when $X$ is a large random set. We show that, as $|X|\to\infty$, the expected fraction of periodic points of $f_r$ tends to $0$ if $r$ is irrational and to $\frac{1}{q}$ if $r=\frac{p}{q}$ is rational with $p$ and $q$ coprime. These results are obtained via more refined statistics of $f_r$ which we compute explicitly in terms of (generalized) Catalan numbers. The motivation for studying $f_r$ comes from Vietoris--Rips complexes, a geometric construction used in computational topology. Our results determine how much one can expect to simplify the Vietoris--Rips complex of a random sample of the circle by removing dominated vertices.
\end{abstract}

\section{Introduction}

\subsection*{Finite cyclic dynamical systems.}
We are interested in a family of finite dynamical systems parametrized by a finite subset $X\subseteq S^1$ of the circle and a real number $0<r\leq 1$. Here $S^1=\RR/\ZZ$ is the circle of unit circumference equipped with the arc-length metric. For $X\subseteq S^1$ finite and $0<r\leq 1$, we define the map $f_r\colon X\to X$ which sends each point $x\in X$ to the furthest element of $X$ within clockwise distance less than $r$ from $x$ (equivalently, within angular distance less than $2\pi r$ from $x$). Iterating $f_r$ gives rise to a discrete time dynamical system on $X$, which we call a \emph{cyclic} dynamical system. We can speak of \emph{periodic} and \emph{non-periodic} points of $f_r$ and the structure and size of the \emph{periodic orbits} of $f_r$ (see Figure~\ref{fig:examples}(a,b)). 

As the set $X\subseteq S^1$ becomes more dense in $S^1$, each cyclic dynamical system $f_r$ can be seen as a discrete approximation of the rigid rotation of the circle by angle $2\pi r$. 

\begin{figure}
\centering
 	\def\svgwidth{1.0\textwidth}
	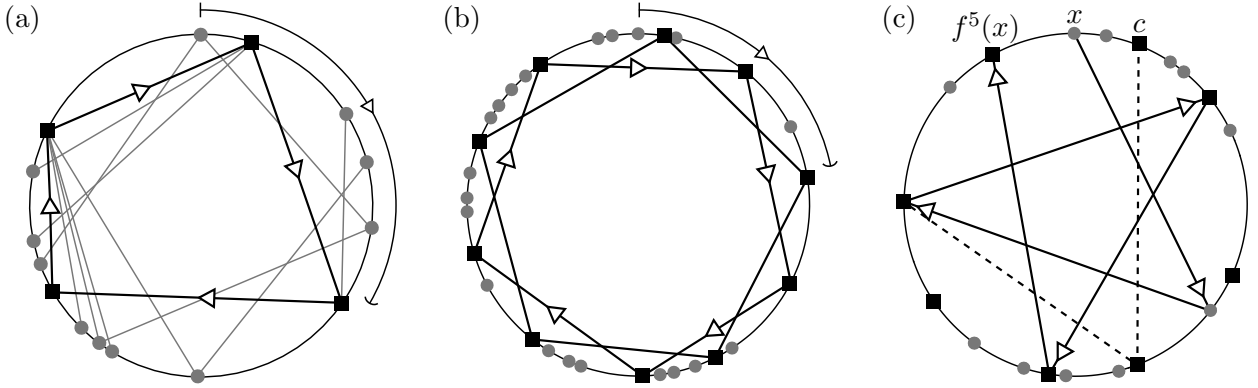
\caption{Random samples $X_n \subset S^1$ (points) and mappings $f_r: X_n \rightarrow X_n$ (edges) which move each point clockwise as far as possible within a distance less than $r$ with (a) $n = 15$, $r = 1/3$, exhibiting a single periodic orbit of length 4, and (b) $n = 30$, $r = 5/23$, exhibiting two periodic orbits of length 5. Periodic points (squares) and orbits are colored black, while non-periodic points (dots) are shaded gray. (c) A 5-swift point $x$, its clockwise-nearest periodic neighbor $c$, and their first 5 forward iterates under $f_r$ for $r = 2/5$, as in the proof of Lemma~\ref{lem:q-swift}.}
\label{fig:examples}
\end{figure}

\subsection*{Random cyclic systems --- main results.}
To get a sample of the circle one can take a random set $X=X_n$ of $n$ points chosen uniformly and independently from $S^1$. What is the asymptotic behaviour of the cyclic dynamical systems $f_r:X_n\to X_n$ as $n\to\infty$? Our main results analyze the number of periodic points (denoted $\per(X,r)$) and the structure of periodic orbits.

\begin{main} Let $X_n$ be a sample of $n$ points chosen uniformly and independently from $S^1$ and let $0<r\leq 1$.
\begin{itemize}
\item The expected fraction of periodic points in $f_r:X_n\to X_n$ is
\[ \lim_{n\to\infty}\frac{\expect[\per(X_n,r)]}{n}=\begin{cases}
0 &\mbox{if }r\mbox{ is irrational},\\
\frac{1}{q} &\mbox{if }r=\frac{p}{q}\mbox{ is rational.\footnotemark }
\end{cases}\]
\footnotetext{Throughout the paper, whenever $r$ is written as $r=\frac{p}{q}$ it is understood that $p,q\in \ZZ$ are relatively prime.}
\item If $r$ is rational, then asymptotically almost surely there is one periodic orbit.
\item If $r$ has irrationality exponent $2$, then the expected number of periodic points satisfies $\expect[\per(X_n,r)]=\Omega(n^{1/2-\varepsilon})$ for any $\varepsilon>0$.
\end{itemize}
\end{main}
The main theorem is a combination of Theorems~\ref{thm:fraction-periodic-points}, \ref{thm:oneorbit}, and \ref{thm:periodic-vertices-lower-bound}. 

Our proofs rely on a more refined count of the non-periodic points of dynamical system $f_r$. We say a non-periodic point $x\in X$ is at \emph{level} $i\ge0$ if $x \in f_r^{i}(X)\setminus f_r^{i+1}(X)$; let $\numlev_i(X,r)$ denote the number of non-periodic points at level $i$. Let $C_i$ be the $i$-th Catalan number, i.e.\ the number of Dyck paths from $(0,0)$ to $(2i,0)$, and let $C_{i,q-2}$ be the number of Dyck paths of height at most $q-2$.  

\begin{4.3}
For $0<r\leq 1$ and $i\geq 0$, the expected fraction of points at level $i$ is
\[ \ell_i(r):=\lim_{n\to\infty}\frac{\expect[\numlev_i(X_n,r)]}{n} = \begin{cases}
\frac{1}{2^{2i+1}}C_i &\mbox{if }r\notin\QQ,\\
\frac{1}{2^{2i+1}}C_{i,q-2} &\mbox{if }r=\frac{p}{q}.
\end{cases} \]
\end{4.3}

\subsection*{Related models.}
For large $n$, the map $f_r \colon X_n \to X_n$ resembles the rigid rotation of $S^1$ by angle $2\pi r$, but the periodic orbit structures of $f_r$ and of a rigid rotation are quite different. Indeed, for rational $r=\frac{p}{q}$ every point of $S^1$ is periodic under the rigid rotation and there are infinitely many periodic orbits, whereas for $f_r$ only about $\frac{1}{q}$-fraction of the points of $X_n$ are periodic and asymptotically there is a single orbit. 

Another way to produce a random endomorphism of an $n$-element set is to select $\phi \colon [n]\to [n]$ uniformly from the set of all $n^n$ possible maps. Expectations of various statistics of $\phi$ are known: asymptotically it has $\sqrt{\frac{\pi n}{2}}$ periodic points in $\mathrm{ln}\,n$ orbits \cite[Theorem 14.33]{bollobas2001random}, and the  fraction of points not in the image of $\phi$ is $(1-\frac{1}{n})^n\approx e^{-1}$. This can be compared with our model. First, $f_{p/q}$ has asymptotically a higher order of $\frac{1}{q}n$ periodic points, all in a single orbit. If $r\notin\QQ$ then we provide a lower bound $\Omega(n^{1/2-\varepsilon})$ on the number of periodic points of $f_r$, which is of similar order of magnitude as for $\phi$. The expected fraction of points not in the image of $f_r:X_n\to X_n$ approaches $1/2$ for all $r<1$.

\subsection*{Application and motivation: Vietoris--Rips complexes.}
Our initial motivation for considering the cyclic dynamical system $f_r$ was to study Vietoris--Rips complexes of finite subsets of the circle. Given a metric space $X$ and a scale parameter $r>0$, the Vietoris--Rips simplicial complex $\vrc{X}{r}$ has $X$ as its vertex set and has a finite subset $\sigma \subseteq X$ as a simplex whenever the diameter of $\sigma$ is less than $r$.
Vietoris--Rips complexes have been used to approximate the shape of a space from a finite sample \cite{Latschev2001}.
In Proposition~\ref{prop:vrchtpy} we leverage the work of \cite{AA-VRS1, AAFPP-J} to show that for $X \subseteq S^1$ finite, the homotopy type of $\vrc{X}{r}$ is determined by the periodic orbits of dynamical system $f_r\colon X\to X$.

In applications of computational topology \cite{EdelsbrunnerHarer, Carlsson2009} one would like to compute with Vietoris--Rips complexes, even though the number of simplices of $\vrc{X}{r}$ can be exponential in $|X|$. One way to simplify a simplicial complex without changing its homotopy type is to remove dominated vertices (this operation is also called a dismantling, fold, elementary strong collapse, or LC reduction \cite{BabsonKozlov2006,BarmakStar2013,Matouvsek2008}). If $X$ is chosen via a random process, how much can one expect to simplify $\vrc{X}{r}$ by removing dominated vertices? We use the map $f_r$ to provide an answer when $X$ is chosen uniformly at random from $S^1$: the expected number of remaining vertices is $o(|X|)$ for $r$ irrational and approximately $\frac{1}{q}|X|$ for $r=\frac{p}{q}$ rational (Theorem~\ref{corollary:fraction-core-points}).

\section{Preliminaries on cyclic dynamical systems}

\subsection{Definitions and basic facts}
\label{SS:Definitions-and-basic-facts}
Let $S^1=\RR/\ZZ$ be the circle of unit circumference equipped with the arc-length metric. For $x,y\in S^1$ we denote the clockwise distance from $x$ to $y$ by $\ddist(x,y)$, and we denote the closed clockwise arc from $x$ to $y$ by $[x,y]_{S^1}$. Open and half-open arcs are denoted similarly.

We give the precise definition of the main object of study in this paper.
\begin{definition}
\label{def:dyngraph}
Let $X\subseteq S^1$ be finite and let $0<r\leq 1$. We define the map $f_r\colon X\to X$ by setting $f_r(x)=\arg\max_{y\in X\cap[x,x+r)_{S^1}}\{\ddist(x,y)\}$. That is, $f_r(x)$ is the element of $X\cap[x,x+r)_{S^1}$ which is clockwise furthest from $x$.
\end{definition}
One could also use the closed arc $[x,x+r]_{S^1}$; this distinction is irrelevant in the random case.

We also refer to $f_r \colon X \to X$ as a dynamical system, by which we mean the discrete time dynamical system generated by $f_r$. It is easy to check that $f_r$ preserves the (weak) cyclic ordering of points on $S^1$ and, for this reason, we refer to $f_r$ as a \emph{cyclic} dynamical system.

A point $x\in X$ is \emph{periodic} if $f_r^i(x)=x$ for some $i\ge1$; otherwise, $x\in X$ is \emph{non-periodic}. If $x$ is periodic then we refer to $\{f_r^i(x)~|~i\ge0\}$ as a \emph{periodic orbit} with \emph{length} equal to the smallest $i\ge1$ such that $f_r^i(x)=x$.  The \emph{core} of a dynamical system is the subsystem restricted to the union of all periodic points.

The \emph{winding number} $\owind(x_1,\ldots,x_s)$ of a sequence of points $x_1,\ldots,x_s\in S^1$ is the number of times the closed walk $x_1\to x_2\to\cdots\to x_s\to x_1$ wraps around the circle; we have $\owind(x_1,\ldots,x_s)=
\sum_{i=1}^{s}\ddist(x_i,x_{i+1})$ where $x_{s+1}=x_1$.
A periodic orbit of length $\olen$ containing $x\in X$ has winding number $\owind(x,f_r(x),\ldots,f_r^{\olen-1}(x))$.

\begin{example}
For $0\leq k<n$ let $\mathrm{Reg}_n^k$ denote the ``regular'' cyclic dynamical system with domain $X=\{\frac{0}{n},\frac{1}{n},\ldots,\frac{n-1}{n}\}\subseteq S^1$ and with $r=\frac{k}{n}+\varepsilon$ for some $0<\varepsilon<\frac{1}{n}$, so that $f_r(\frac{i}{n})=\frac{i+k}{n}\ \mathrm{mod}\ 1$. All points of this system are periodic. If $d=\mathrm{gcd}(k,n)$ then $\mathrm{Reg}_n^k$ has $d$ periodic orbits, each of length $n/d$ and winding number $k/d$. It follows from Lemma~\ref{lem:same-orbits} below that the core of any cyclic dynamical system is isomorphic to $\mathrm{Reg}_n^k$ for some $0\leq k< n$. For the systems in Figure~\ref{fig:examples}(a,b) the cores are $\mathrm{Reg}_4^1$ and $\mathrm{Reg}_{10}^2$, respectively.
\end{example}

Consider a periodic orbit of length $\olen$ and winding number $w$. Necessarily $\olen$ and $w$ are relatively prime and $w<\olen$. If the points along this periodic orbit are cyclically ordered as $x_0,\ldots, x_{\olen-1}$, then it follows that $f_r(x_i)=x_{(i+w) \md \olen}$ for all $i$, since $f_r$ is cyclic and bijective on the points of the orbit.
It is not hard to show that any two periodic orbits are interleaved; this yields the following lemma.

\begin{lemma}
\label{lem:same-orbits}
All periodic orbits of $f_r\colon X\to X$ have the same length $\olen$ and the same winding number $w$.
\end{lemma}


\begin{definition}
\label{def:wf}
The common value of $\frac{w}{\ell}$ for all periodic orbits of $f_r:X\to X$ is called the \emph{winding fraction} and is denoted $\wf(X,r)$.
\end{definition}
Since $X$ is finite, the iterates of any point under $f_r$ eventually reach a periodic orbit, after which they wind $w$ times around $S^1$ every $\olen$ steps. It follows that for \emph{any} $x\in X$ we have 
\begin{equation}
\label{eq:wf-formula}
\wf(X,r)=\lim_{n\to\infty}\frac{1}{n}\sum_{i=0}^{n-1}\ddist(f_r^i(x),f_r^{i+1}(x)).
\end{equation}
In particular, since $\ddist(x,f_r(x))< r$ for all $x\in X$, we immediately conclude $\wf(X,r)< r$.
\begin{definition}
\label{def:level}
The \emph{level} of a non-periodic point $x\in X$ is the unique $i$ such that $x \in f_r^{i}(X)\setminus f_r^{i+1}(X)$. We denote the number of non-periodic points at level $i\geq 0$ by $\numlev_i(X,r)$, and we write $\per(X,r)$ for the number of periodic points of the system.
\end{definition}
By convention we set $f_r^0=\mathrm{id}_X$, so the vertices at level $0$ are those not in the image of $f_r$. 
 Of course, we have
$
\per(X,r)=|X|-\sum_{i\geq 0}\numlev_i(X,r).
$

\subsection{Swift points}
\label{sect:swift}

When $r=\frac{p}{q}$ is rational, we introduce a useful tool  for identifying some of the periodic points of $f_r \colon X \to X$, which we will then count in Theorem~\ref{thm:fraction-periodic-points}.

\begin{definition}\label{def:q-swift}
For $r=\frac{p}{q}$, a point $x\in X$ is \emph{$q$-swift} if $\owind(x,f_r(x),\ldots,f_r^q(x))=p$ and the open arc $(f_r^q(x),x)_{S^1}$ does not contain any point of $X$.
\end{definition}

Each step under $f_{p/q}$ moves points by less than $\frac{p}{q}$, so during $q$ steps we wrap around the circle at most $p$ times. Intuitively, $x$ is $q$-swift if in $q$ steps it moves as far as possible, wrapping around $S^1$ almost $p$ times and ending at the nearest counter-clockwise neighbor of $x$; see Figure~\ref{fig:examples}(c).

\begin{lemma}\label{lem:q-swift}
Suppose $r=\frac{p}{q}$ and $f_r\colon X\to X$ contains a $q$-swift point $x$. Then the point $f_r^q(x)$ is periodic, and the system has a single periodic orbit whose winding number $w$ and length $\olen$ satisfy $\olen p-wq=1$.
\end{lemma}

Definition~\ref{def:q-swift} is really a special case of a more general notion (Definition~\ref{def:qi-swift}) useful for identifying periodic points; a $q$-swift point will be the same thing as a $(q,0)$-swift point. Lemma~\ref{lem:q-swift} follows from the more general Lemma~\ref{lem:qi-swift}.

\begin{definition}\label{def:qi-swift}
For $r=\frac{p}{q}$, a point $x\in X$ is \emph{$(q,i)$-swift} if $\owind(x,f_r(x),\ldots,f_r^q(x))=p$ and $x$ is the first clockwise point after the preimage set $f_r^{-i}(f_r^{q+i}(x))$.
\end{definition}

\begin{lemma}
\label{lem:qi-swift}
Suppose $r=\frac{p}{q}$ and $f_r\colon X\to X$ contains a $(q,i)$-swift point $x$. Then the point $f_r^{q+i}(x)$ is periodic, and the system has a single periodic orbit whose winding number $w$ and length $\olen$ satisfy $\olen p-wq=1$.
\end{lemma}

\begin{proof}
Let $c$ be the first periodic point equal to or clockwise after a $(q,i)$-swift point $x$, as in Figure~\ref{fig:examples}(c). Since $f_r$ preserves the cyclic ordering of $X$, by comparing iterates of $c$ and $x$ we see $\owind(c,f_r(c),\ldots,f_r^q(c))=p$. This follows since $f_r^q(c)\in[f_r^q(x),c)_{S^1}$, where the interval is half open because $c$ cannot wrap $p$ times around $S^1$ in $q$ steps. By definition of $c$, $(x,c)_{S^1}$ contains no periodic points and hence no iterates of $c$. Thus, the only possible periodic points in $[f_r^q(x),c)_{S^1}$ are in $[f_r^q(x),x)_{S^1}$. That means 
\[ f_r^q(c)\in X\cap [f_r^q(x),x)_{S^1}\subseteq f_r^{-i}(f_r^{q+i}(x)), \]
where the last inclusion follows from the observation that if $y \in X\cap [f_r^q(x),x)$ but $y \notin f_r^{-i}(f_r^{q+i}(x))$, then $y$ must be clockwise after $f_r^{-i}(f_r^{q+i}(x))$, which violates the assumption that $x$ is $(q,i)$-swift. It follows from $f_r^q(c)\in f_r^{-i}(f_r^{q+i}(x))$ that $f_r^{q+i}(c)=f_r^{q+i}(x)$ and so  $f_r^{q+i}(x)$ is periodic.

If $c'\in X\cap [f_r^q(x),x)_{S^1}$ is any periodic point then $f_r^i(c')=f_r^{q+i}(x)=f_r^{i}(f_r^q(c))$. Since $f_r$ restricted to the set of periodic points is a bijection, we get $c'=f_r^q(c)$. It follows that the arc $(f_r^q(c),x)_{S^1}$, and consequently also the arc $(f_r^q(c),c)_{S^1}$, contain no periodic points.

We showed that $f_r^q(c)$ and $c$ are two consecutive (in the clockwise direction) periodic points of $f_r$, and they belong to same orbit. That implies $f_r$ has a single periodic orbit, since we observed that multiple periodic orbits must be interleaved. It follows that the core of $f_r$ is isomorphic to $\mathrm{Reg}_\olen^w$ for $\olen, w$ coprime. The identification $c=0$, which implies $f_r^q(c)=\frac{\olen-1}{\olen}$, combined with fact that $f_r^q(c)$ is obtained from $c$ in $q$ steps of winding number $p$, translates to the equation $q \frac{w}{\olen}=p-\frac{1}{\olen}$, i.e.\ $\olen p-wq=1$.
\end{proof}

\section{Classifying and counting non-periodic points}

\subsection{Probabilistic preliminaries}
\label{sect:prob}
We now describe the probabilistic background for our calculations. Let $X_n\subseteq S^1$ be a set of $n$ points sampled independently and uniformly at random. Fix a point $p\in S^1$, and let $P_n=\min\{\ddist(p,x)~:~x\in X_n\}$ be the random variable measuring the distance from $p$ to the nearest clockwise point of $X_n$. For $0<d\leq 1$ we have $\prob[P_n>d]=(1-d)^n$. As a consequence, for every fixed $t>0$ we get
\[ \prob[nP_n>t]=\prob[P_n>\tfrac{t}{n}]=\max\{0,(1-\tfrac{t}{n})\}^n\to e^{-t}\quad\mathrm{as}\quad n\to\infty, \]
meaning that the sequence of random variables $nP_n$ converges in distribution to the exponential random variable $T$ on $[0,\infty)$. The exponential distribution has tail $\prob[T> t]=e^{-t}$ and density $e^{-t}$ for $t\geq 0$.

Now suppose we have multiple marked points $p^1,\ldots, p^k\in S^1$, each with its own sequence of random variables $P^j_n=\min\{\ddist(p^j,x)~:~x\in X_n\}$. For any fixed $n$ the variables $P^1_n,\ldots,P^k_n$ are not independent. This observation is clear if $n$ is taken to be small, say $n<k$, in which case there must be some $p^i$ and $p^j$ whose nearest clockwise neighbor is the same point $x \in X_n$. However, as $n\to\infty$ the nearest clockwise neighbors to each $p^j$ ($j=1,\ldots,k$) are all distinct with probability tending to $1$. We conclude that the joint distribution of $(nP^1_n,\ldots,nP^k_n)$ converges to a tuple of $k$ independent exponential random variables $(T_1,\ldots,T_k)$. In particular, for any measurable event $S\subseteq \RR_+^k=\{x=(x_1,\ldots,x_k)~:~x_i\geq 0\}$ we have
\[ \lim_{n\to\infty}\prob[(nP^1_n,\ldots,nP^k_n)\in S]=\prob[(T_1,\ldots,T_k)\in S]=\int_S e^{-\|x\|_1}dx. \]
Here $\|x\|_1=\sum_{i=1}^k x_i$ since $x_i\geq 0$.

In Sections~\ref{SS:determining-the-level}--\ref{SS:qswiftcount} we use a more complicated variant where the marked points $p^1,\ldots,p^k$ are not fixed in advance, and instead $p^j$ depends on $p^{j-1}$ and its nearest neighbor in $X_n$. This will require longer calculations which we defer to Appendix~\ref{append:independence}. However, the intuition remains the same: as long as $k$ is fixed and $n\to\infty$, the principle of asymptotic independence discussed above will still apply because the $p^j$s will be sufficiently separated.

\subsection{Example: number of points in the image of $f_r$}\label{SS:example12}
We will illustrate the above method by arguing that the expected fraction of points in the image of $f_r$ for $r<1$ approaches $\frac12$ as $n\to\infty$. Fix $x_0\in X_n$; we want to show that $\lim_{n\to\infty}\prob[x_0\in f_r(X_n)]=\frac12$. Let $z$ be the distance from $x_0$ to the nearest clockwise point of $X_n$ (call that point $y_0$), and let $w$ be the distance from $x_0-r$ to the nearest clockwise point of $X_n$ (call that point $x_1$). If $w<z$ then $f_r(x_1)=x_0$, and hence $x_0\in f_r(X_n)$. If $w>z$ then $f_r(x_1)$ is clockwise at least as far as $y_0$, and moreover $x_0\notin f_r(X_n)$. The event $w=z$ has probability $0$. As discussed above, for large $n$ the distributions of $nw$ and $nz$ can be approximated by a pair of identical independent exponential random variables; hence 
$\lim_{n\to\infty}\prob[w<z]=
\lim_{n\to\infty}\prob[w>z]=\frac12$
by symmetry. It follows that $\lim_{n\to\infty}\prob[x_0\in f_r(X_n)]=\lim_{n\to\infty}\prob[x_0\notin f_r(X_n)]=\frac12$. This example is the specific case of Theorem~\ref{thm:levels} when $i=0$ and $r<1$.

\subsection{Determining the level of a non-periodic point}\label{SS:determining-the-level}
Let $X\subset S^1$ be any finite subset (not necessarily random) and fix $x_0\in X$. We define a sequence $x_0, y_0, x_1, y_1,\ldots\in X$ as follows (see Figure~\ref{fig:variables}(a)). First, let $y_0$ be the nearest clockwise neighbor of $x_0$. Inductively, we define $x_j$ as the nearest clockwise point of $X$ strictly after $x_{j-1}-r$ and $y_j$ as the nearest clockwise point of $X$ strictly after $y_{j-1}-r$. The significance of these points follows from the next lemma.

\begin{lemma}
\label{lem:preimage}
For $i\geq 0$ we have $f_r^{-i}(x_0)=X\cap [x_i,y_i)_{S^1}$.
\end{lemma}
\begin{proof}
This is true by definition when $i=0$. For the induction step, since $x_i,y_i\in X$ we have
\[ f_r^{-1}(X\cap [x_i,y_i)_{S^1})=X\cap (x_{i-1}-r,y_{i-1}-r]_{S^1} = X\cap[x_{i+1},y_{i+1})_{S^1}. \]
\end{proof}

It is possible that $x_i=y_i$, which happens precisely when $[x_{i-1}-r,y_{i-1}-r)_{S^1}$ contains no points of $X$ (Figure~\ref{fig:variables}(c)). In that case $f_r^{-i}(x_0)=\emptyset$, the point $x_0$ is at level $i-1$, and $x_j=y_j$ for $j\geq i$.

Our purpose in defining the sequence $x_i$ (or similarly, the sequence $y_i$) is actually to understand the sequence of gaps between the points $x_{i-1}-r$ and $x_i$, which we think of as a lag in the potential distance to travel under $f_r$. Relationships between these gaps, made precise in Lemma~\ref{lem:preimage2}, characterize the level of $x_0$.

We define a sequence of distances $z_1,w_1,z_2,w_2,\ldots\in\RR_+$ as follows. We set $z_1=\ddist(x_0,y_0)$ and for $j \geq 1$, we set $w_j=\ddist(x_{j-1}-r,x_j)$ and $z_{j+1}=\ddist(y_{j-1}-r,y_{j})$.

\begin{lemma}
\label{lem:preimage2}
For $i\geq 0$ we have $x_0\in f_r^i(X)$ if and only if $\sum_{j=1}^i w_j\le\sum_{j=1}^i z_j$.
\end{lemma}
\begin{proof}
The relation
\[ \ddist(x_i,y_i)=\ddist(x_{i-1}-r,y_{i-1}-r)+z_{i+1}-w_i=\ddist(x_{i-1},y_{i-1})+z_{i+1}-w_i \]
implies by induction that $\ddist(x_i,y_i)=\sum_{j=1}^{i+1}z_j-\sum_{j=1}^i w_j$ for $i\geq 0$. By Lemma~\ref{lem:preimage},
\[x_0\in f_r^i(X) \quad\Longleftrightarrow\quad x_i\neq y_i \quad\Longleftrightarrow\quad w_i\le\ddist(x_{i-1}-r,y_{i-1}-r)=\ddist(x_{i-1},y_{i-1})=\sum_{j=1}^{i}z_j-\sum_{j=1}^{i-1} w_j, \] 
completing the proof.
\end{proof}

\begin{figure}
\centering
\def\svgwidth{1.0\textwidth}
 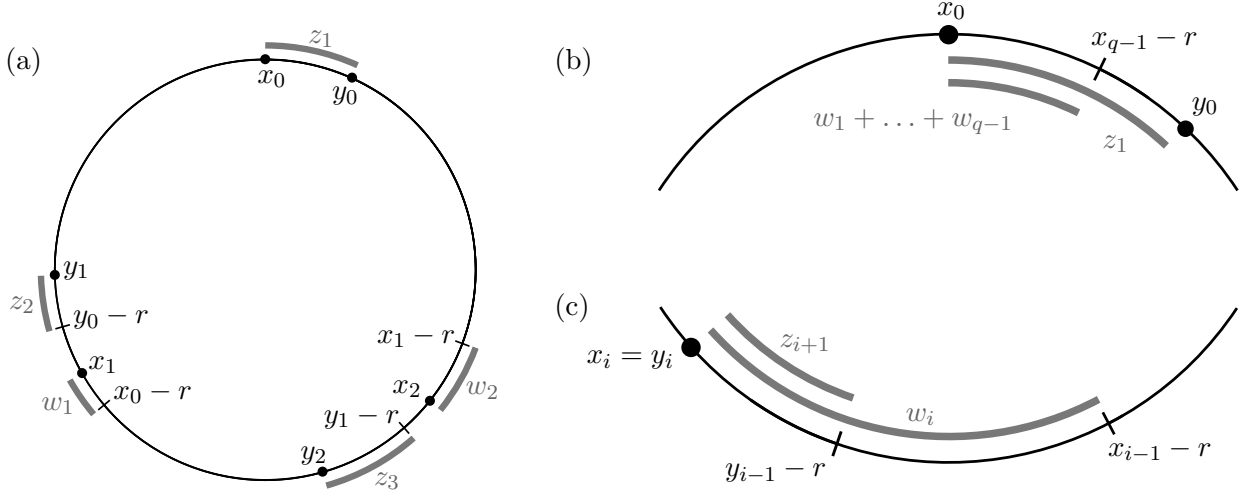
\caption{(a) Points $x_{j}, y_{j}$ and distances $z_j, w_j$. (b) The event when $x_{q-1}-r \in [x_0,y_0)_{S^1}$ and $x_q=y_0$, as considered in the proof of Proposition~\ref{prop:levels} for $r=\frac{p}{q}$. (c) Instance when $x_{0}$ is a level $i-1$ point and $x_i=y_i$.}
\label{fig:variables}
\end{figure}

\subsection{Expected fraction of points at each level}
\label{sect:prob2}
If $X_n$ is a random sample, then the distances $z_j$ and $w_j$ become random variables, which can be used to calculate the probability that an element of $X_n$ is at level $i$ as $n\to\infty$ (Proposition~\ref{prop:levels}). We adopt the convention that points in $\RR_+^{2i+2}$ are labelled as sequences $u=(z_1,\ldots,z_{i+1},w_1,\ldots,w_{i+1})$. If $H$ is a linear inequality with variables $z_j, w_j$, then we let $\neg H$ denote the inverse inequality. For instance, if $H$ is $z_1\ge w_1$ then $\neg H$ is $z_1<w_1$. By an abuse of notation we also let $H$ denote the subset of all points in $\RR_+^{2i+2}$ satisfying the linear inequality $H$. For $1\le j,k \le i+1$, let $H(j,k)$ denote the inequality $z_1+\cdots+z_j \geq w_1+\cdots+w_k$.

\begin{definition}
We define the convex cones $K_i$ and $K_i(q)$ in $\RR^{2i+2}_+$ by the following inequalities:
\begin{equation}
\label{eq:define-k}
\begin{aligned}
K_i: & & H(1,1),\ H(2,2),\ \ldots,\ H(i,i),\ \neg H(i+1,i+1).\\
K_i(q): & & H(1,1),\ H(2,2),\ \ldots,\ H(i,i),\ \neg H(i+1,i+1),\\
& & \neg H(1,q-1),\ \neg H(2,q),\ \ldots,\ \neg  H(i-q+2,i).
\end{aligned}
\end{equation}
\end{definition}
If $i\leq q-2$, then the second row of inequalities defining $K_i(q)$ is empty and $K_i(q)=K_i$.

We will now prove the main technical result of this section, Proposition~\ref{prop:levels}. Two steps in the proof, which use the principle of asymptotic independence introduced in Section~\ref{sect:prob}, are deferred to Appendix~\ref{append:independence}.


\begin{proposition}
\label{prop:levels}
For $0<r\le1$ and $i\geq 0$, the expected fraction of points at level $i$ is
\begin{equation*} \ell_i(r):=\lim_{n\to\infty}\frac{\expect[\numlev_i(X_n,r)]}{n} = \begin{cases}
\int_{K_i}e^{-\|u\|_1}du &\mbox{if }r\notin\QQ,\\
\int_{K_i(q)}e^{-\|u\|_1}du &\mbox{if }r=\frac{p}{q}.
\end{cases}
\end{equation*}
\end{proposition}

\begin{proof} 
Fix $i\geq 0$ and fix some $x_0\in X_n$. We want to compute the probability that $x_0$ is at level $i$; this occurs when $x_0\in f_r^j(X_n)$ for $1\leq j\leq i$ and $x_0\notin f_r^{i+1}(X_n)$. By Lemma~\ref{lem:preimage2} this is equivalent to the following conditions on $u=(z_1,\ldots,z_{i+1},w_1,\ldots,w_{i+1})$:
\begin{equation}
\label{eq:leveli}
z_1\ge w_1,\ \ldots,\ z_1+\cdots+z_i\ge w_1+\cdots+w_i,\quad z_1+\cdots+z_{i+1}<w_1+\cdots+w_{i+1},
\end{equation}
or in other words, to the condition $u\in K_i$. We have
\[ \ell_i(r)=\lim_{n\to\infty}\prob[x_0\mbox{ at level }i]=\lim_{n\to\infty}\prob[u\in K_i].\]
Moreover, as long as $i$ is fixed, we can pick $n$ sufficiently large so that, with probability arbitrarily close to $1$, the points $x_0,x_1,x_2,\ldots,x_i$ are arbitrarily close to $x_0,x_0-r,x_0-2r,\ldots,x_0-ir$ and the arcs $[x_0,y_0]_{S^1},[x_0-r,y_1]_{S^1},\ldots,[x_i-r,y_{i+1}]_{S^1}$ are arbitrarily short (Figure~\ref{fig:variables}).

When $r\notin\QQ$, by Lemma~\ref{lem:K_i} we have
\[ \lim_{n\to\infty}\prob[u\in K_i]=\lim_{n\to\infty}\prob[nu\in K_i]=\int_{K_i}e^{-\|u\|_1}du, \]
where the first equality follows since $K_i$ is a cone. The intuition behind Lemma~\ref{lem:K_i} is that, since $r\notin\QQ$, the points $x_0,x_0-r,x_0-2r,\ldots,x_0-ir$ are all distinct and, for $n$ sufficiently large, the arcs $[x_0,y_0]_{S^1},[x_0-r,y_1]_{S^1},\ldots,[x_i-r,y_{i+1}]_{S^1}$ are almost surely pairwise disjoint. Therefore, the distances $z_j,w_j$ can be treated as asymptotically independent (they depend on the local structure of the sample in sufficiently separated locations on the circle). By the discussion of Section~\ref{sect:prob}, $nu$ converges to a $(2i+2)$-tuple of independent exponential variables as $n\to\infty$.

Next let $r=\frac{p}{q}$ be rational. The points $x_0,x_0-r,x_0-2r,\ldots,x_0-ir$ are no longer distinct, and therefore we cannot directly treat $u$ as a sequence of independent random variables. Indeed, the event $x_j=y_{j-q}$ now happens with positive probability for any $j\geq q$ (see Figure~\ref{fig:variables}(b) for the case $j=q$), leading to a linear relation between the entries of $u$ and excluding the possibility of asymptotic independence. However, we will show that those  situations produce periodic points, rather than points at level $i$, and can be excluded using the additional inequalities defining $K_i(q)$, so that we can again restrict to disjoint intervals.

We claim that if $u$ is a sequence of gaps from points on the circle, then $u\in K_i$ if and only if $u\in K_i(q)$. First, suppose that $x_{q-1}-r\in[x_0,y_0)_{S^1}$, as in Figure~\ref{fig:variables}(b). Then $x_q=y_0$ and hence $f_r^q(y_0)=x_0$, so $y_0$ is $q$-swift and $x_0$ is periodic (hence not at level $i$) by Lemma~\ref{lem:q-swift}. The condition $x_{q-1}-r\in[x_0,y_0)_{S^1}$ is equivalent to $w_1+\cdots+w_{q-1}<z_1$, which is $H(1,q-1)$. It follows that $\neg H(1,q-1)$ is a necessary condition for $x_0$ to be at level $i$; this negation also implies that the arcs $[x_0,y_0]_{S^1}$ and $[x_q,y_q]_{S_1}$ are disjoint.

Under the assumption $\neg H(1,q-1)$ we next look at the relative order of $y_0-r$, $x_q-r$ and $y_1$. If $x_q-r\in[y_0-r,y_1)_{S^1}$ then $x_{q+1}=y_1$ is the first point of $X_n$ clockwise after preimage set $f_r^{-1}(x_0)=X\cap[x_1,y_1)_{S^1}$. Moreover,  $f_r^{q+1}(y_1)=x_0$. This means that $y_1$ is $(q,1)$-swift and $x_0$ is periodic by Lemma~\ref{lem:qi-swift}. Excluding this situation requires the condition $\neg H(2,q)$, which also assures that $[x_1,y_1]_{S^1}$ and $[x_{q+1},y_{q+1}]_{S^1}$ are disjoint.

Proceeding in a similar fashion, we see that the second row of conditions defining $K_i(q)$ in \eqref{eq:define-k} guarantees that $x_0$ is not a $(q+j)$-th image of some $(q,j)$-swift point, which is a necessary condition for $x_0$ to be at level $i$. It follows that if $u$ is a sequence of gaps from points on the circle, then $u\in K_i$ if and only if $u\in K_i(q)$. We have
\[ \ell_i(r)=\lim_{n\to\infty}\prob[u\in K_i]=\lim_{n\to\infty}\prob[u\in K_i(q)]=\lim_{n\to\infty}\prob[nu\in K_i(q)]=\int_{K_i(q)}e^{-\|u\|_1}du, \]
where the last equality is proven in Lemma~\ref{lem:K_i(q)}. The intuition behind Lemma~\ref{lem:K_i(q)} is that conditioning on $K_i(q)$ means that the distances $w_j,z_j$ can be treated as asymptotically independent, and so when restricted to $K_i(q)$ the random variable $nu$ again converges to $2i+2$ independent exponential variables.
\end{proof}

\subsection{Periodic points with $(q,i)$-swift certificates}
\label{SS:qswiftcount}

Using the last proof we can also compute, for $r=\frac{p}{q}$,  the asymptotic expected fraction of periodic points whose periodicity can be certified using Lemma~\ref{lem:qi-swift}. This will be used to get a lower bound on the expected fraction of periodic points in the proof of Theorem~\ref{thm:fraction-periodic-points}.

To make this more precise, we say that a periodic point $x_0$ is \emph{of swiftness type $i$} if $i$ is the smallest index such that $x_0=f_r^{q+i}(y)$ for some $(q,i)$-swift point $y$. Clearly a periodic point can be of swiftness type $i$ for at most one value of $i$, and it is also possible\footnote{Indeed, let $X=\{0,\frac{1}{3},\frac{2}{3}$\} and let $r=\frac{p}{q}=\frac{4}{9}$. Then $f_r$ has one periodic orbit of length three. In this example $f_r$ is invertible, and hence $f_r^{-i}(f_r^{q+i}(x)) = f_r^q(x) = x$ for all $x \in X$ and $i\ge0$. Since the clockwise first neighbor of $x$ is $f_r(x)\neq x$, no point is $(q,i)$-swift for any $i$ and no periodic point is of swiftness type $i$ for any $i$.} that a periodic point is not of swiftness type $i$ for any $i$. For $i\ge0$ let $\swi_i(X,\frac{p}{q})$ denote the number of periodic points of $f_{p/q}:X\to X$ of swiftness type $i$. Let $S_i(q)\subseteq \RR^{2i+2q-2}_+$ be the cone defined by the following inequalities in variables $u=(z_1,\ldots,z_{i+q-1},w_1,\ldots,w_{i+q-1})$:
\begin{equation}
\label{eq:define-s}
\begin{aligned}
S_i(q): & & H(1,1),\ H(2,2),\ \ldots,\ H(i+q-1,i+q-1),\\
& &  \neg H(1,q-1),\ \neg H(2,q),\ \ldots,\ \neg  H(i,i+q-2),\ H(i+1,i+q-1).
\end{aligned}
\end{equation}

The next result now follows from the proof technique of Proposition~\ref{prop:levels} and from Lemma~\ref{lem:S_i(q)}.

\begin{proposition}
\label{prop:swift-witness}
For $r=\frac{p}{q}$ and $i\geq 0$, the expected fraction of periodic points of swiftness type $i$ is
\begin{equation*} sw_i(p/q):=\lim_{n\to\infty}\frac{\expect[\swi_i(X_n,\tfrac{p}{q})]}{n}=\lim_{n\to\infty}\prob[u\in S_i(q)]=\int_{S_i(q)}e^{-\|u\|_1}du. 
\end{equation*}
\end{proposition}

We have established the limiting probabilities that a point is at level $i$ or of swiftness type $i$ in the forms of integrals. It now remains to compute these integrals explicitly, which will allow us to compute the expected fraction of periodic points. 

\section{Fraction of periodic points and Catalan numbers}
\label{sect:catalan}

In this section we use a bijective argument to compute the exponential integrals appearing in Propositions~\ref{prop:levels} and \ref{prop:swift-witness}.
Recall that a \emph{Dyck path} of order $i$ is a path in $\RR^2$ from $(0,0)$ to $(2i,0)$ with $i$ up-steps of the form $(1,1)$ and $i$ down-steps of the form $(1,-1)$, in which $y\geq 0$ for all points $(x,y)$ on the path. The number of Dyck paths of order $i$ is the Catalan number $C_i=\frac{1}{i+1}{2i\choose i}$. We need two generalizations of Catalan numbers:
\begin{itemize}
\item[(i)] $C_{i,h}$ is the number of paths starting at $(0,0)$ with $i$ up-steps and $i$ down-steps, such that $0\leq y\leq h$ for all $(x,y)$ on the path. These are Dyck paths of height bounded by $h$. Such paths end at $(2i,0)$.
\item[(ii)] $C'_{i,h}$ is the number of paths starting at $(0,0)$ with $i+h$ up-steps and $i$ down-steps, such that $0\leq y\leq h$ for all $(x,y)$ on the path. Such paths end at $(2i+h,h)$.
\end{itemize}

\begin{proposition}\label{prop:volumeDyck}
For $i\ge 0$ and $q\geq 2$ the exponential integrals over the cones in \eqref{eq:define-k} and \eqref{eq:define-s} are
\[ \int_{K_i}e^{-\|u\|_1}du=\frac{1}{2^{2i+1}}C_i, \quad  \int_{K_i(q)}e^{-\|u\|_1}du=\frac{1}{2^{2i+1}}C_{i,q-2},\quad \int_{S_i(q)}e^{-\|u\|_1}du=\frac{1}{2^{2i+q-1}}C'_{i,q-2}. \]
\end{proposition}

The $K_i$ case can be proven using the ideas and techniques of \cite[Appendix B]{cantarella2013symplectic} and the references within. In order to handle all three cases of $K_i$, $K_i(q)$, and $S_i(q)$ simultaneously, we combinatorially split these cones into $C_i$ (resp.\ $C_{i,q-2}$ and $C'_{i,q-2}$) smaller parts. We compute the integral over each part using the following lemma, verified by a straightforward computation.

\begin{lemma}\label{lem:integral}
For any $k>m\geq 1$ we have
$ I(k,m):= \int_{0\leq t_1\leq\cdots\leq t_k} e^{-t_{m}-t_k}\ dt_1 \cdots dt_k= \frac{1}{2^{m}}. $
\end{lemma}


\begin{figure}
\centering
 	\def\svgwidth{1.0\textwidth}
	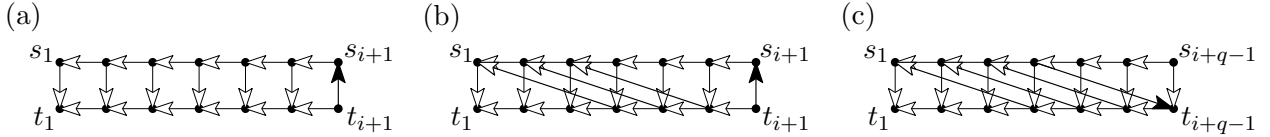
\caption{Three types of posets whose linear extensions are counted by Catalan numbers $C_i$, $C_{i,q-2}$ and $C'_{i,q-2}$, respectively, in the proof of Proposition~\ref{prop:volumeDyck}. Black arrow tips emphasize a relation opposite to the preceding ones. }
\label{fig:ladder}
\end{figure}

\begin{proof}[Proof of Proposition~\ref{prop:volumeDyck}]
We need to compute
$\int_{K_i} e^{-\|u\|_1}\ du = \int_{K_i} e^{-\sum_{j=1}^{i+1}z_j-\sum_{j=1}^{i+1}w_j}\ dz_1\cdots dw_{i+1}.$
The linear substitution
\begin{eqnarray*}
s_1=z_1,\ s_2=z_1+z_2,\ \ldots,\ s_{i+1}=z_1+\cdots+z_{i+1},\\
t_1=w_1,\ t_2=w_1+w_2,\ \ldots,\ t_{i+1}=w_1+\cdots+w_{i+1},
\end{eqnarray*}
whose determinant is one, reduces this integral to $\int_{K_i'}e^{-s_{i+1}-t_{i+1}}\ ds_1\cdots dt_{i+1}$, where $K_i'$ is the cone determined by the inequalities
\begin{equation}\label{eq:orderings}
\begin{aligned}
0\leq s_1\leq\cdots\leq s_{i+1},\quad
0\leq t_1\leq\cdots\leq t_{i+1},\quad 
s_1\geq t_1,\ \ldots,\ s_i\geq t_i,\quad  s_{i+1}< t_{i+1}.
\end{aligned}
\end{equation}
These relations are depicted in Figure~\ref{fig:ladder}(a). Consider all possible ways to impose a total order on the variables $s_i$, $t_i$ satisfying \eqref{eq:orderings}; necessarily $s_{i+1}\leq t_{i+1}$ are the two largest. The remaining variables $s_1,\ldots,s_i,t_1,\ldots,t_i$ can be ordered in $C_i$ ways, which one can see by thinking of the $t_j$ as up-steps and of the $s_j$ as down-steps.\footnote{One of the (many) standard interpretations of $C_i$ is the number of linear extensions of the ladder poset $2\times i$ \cite[Ex. 6.19.(aaa)]{stanley1999enumerative}.
} Therefore $K_i'$ splits into $C_i$ smaller cones, one for each total ordering of $2i+2$ variables, and the integral over each is of the form $I(2i+2,2i+1)$. It follows that $\int_{K_i}e^{-\|u\|_1}du=\frac{1}{2^{2i+1}}C_i$.

The computation for $K_i(q)$ proceeds similarly, but from \eqref{eq:define-k} we get additional constraints $s_j< t_{j+q-2}$ for $1\le j\le i-q+2$. The number of total orderings (Figure~\ref{fig:ladder}(b)) is now $C_{i,q-2}$, giving $\int_{K_i(q)}e^{-\|u\|_1}du=\frac{1}{2^{2i+1}}C_{i,q-2}$.

The same substitution translates the inequalities \eqref{eq:define-s} defining $S_i(q)$ into the constraints
\begin{equation}\label{eq:orderings2}
\begin{aligned}
0\leq s_1\leq\cdots\leq s_{i+q-1},\quad 
0\leq t_1\leq\cdots\leq t_{i+q-1},\quad
s_1\geq t_1,\ \ldots,\ s_{i+q-1}\geq t_{i+q-1}\\
s_1< t_{q-1},\ \ldots,\ s_i< t_{i+q-2},\quad s_{i+1}\geq t_{i+q-1}.
\end{aligned}
\end{equation}
In every total order compatible with \eqref{eq:orderings2} (Figure~\ref{fig:ladder}(c)) the biggest variables are $t_{i+q-1}\leq s_{i+1}\leq \cdots\leq s_{i+q-1}$, and the other variables can be ordered in $C'_{i,q-2}$ ways (there remain $i$ of the $s$ and $i+q-2$ of the $t$ variables, with constraints of the form $s_j< t_{j+q-2}$). The integral of $e^{-s_{i+q-1}-t_{i+q-1}}$ over each resulting subcone is of type $I(2i+2q-2,2i+q-1)$ and so the total integral is $\frac{1}{2^{2i+q-1}}C'_{i,q-2}$.
\end{proof}

\begin{theorem}\label{thm:levels}
For $0<r\leq 1$ and $i\geq 0$, the expected fraction of points at level $i$ is
\[ \ell_i(r)=\lim_{n\to\infty}\frac{\expect[\numlev_i(X_n,r)]}{n} = \begin{cases}
\frac{1}{2^{2i+1}}C_i &\mbox{if }r\notin\QQ,\\
\frac{1}{2^{2i+1}}C_{i,q-2} &\mbox{if }r=\frac{p}{q}.
\end{cases} \]
\end{theorem}

\begin{proof}
This follows immediately from Propositions~\ref{prop:levels} and \ref{prop:volumeDyck}.
\end{proof}

\begin{remark}
If $i\le q-2$, then $C_{i,q-2}=C_i$ and hence $\ell_i(\frac{p}{q})=\ell_i(r)$ for any irrational $r$.
 Also, for any $r<1$, the expected fraction of points at level $0$ (i.e.\ not in the image of $f_r$) as $n\to\infty$ is $\ell_0(r)=\frac12$.
\end{remark}
\begin{remark}
We have $\ell_0(\frac12)=\frac12$, and for $i>0$ we have $C_{i,0}=0$ and hence $\ell_i(\frac12)=0$. This is a special feature of $r=\frac12$ not related to randomness. We leave it as an exercise that if $X\subseteq S^1$ is a finite set without antipodal pairs then $\numlev_i(X,\frac12)=0$ for $i>0$, i.e.\ every point of the system $f_{1/2}\colon X\to X$ is either periodic or at level $0$.
\end{remark}

The generating function of Catalan numbers is $C(x):=\sum_{i=0}^\infty C_i x^i=\frac{2}{1+\sqrt{1-4x}}$. This well-known fact follows from the recurrence relation $C_0=1$ and $C_{n+1}=\sum_{j=0}^{n} C_j C_{n-j}$. The recurrence relations for $C_{i,h}$ and $C'_{i,h}$ can be derived with similar arguments, yielding $C_{0,h}=1,\ C_{n+1,h}=\sum_{j=0}^n C_{j,h-1}C_{n-j,h}$ and $C'_{0,h}=1,\ C'_{n+1,h}=\sum_{j=0}^{n+1}C'_{j,h-1}C_{n+1-j,h}$, $n\geq 0$.
By standard manipulations we get the following relations between the generating functions $C_h(x):=\sum_{i=0}^\infty C_{i,h}x^i$ and $C'_h(x):=\sum_{i=0}^\infty C'_{i,h}x^i$:
\begin{equation}
\label{eq:recu-gen}
C_0(x)=C'_0(x)=1,\quad C_h(x)=\frac{1}{1-xC_{h-1}(x)},\quad C'_h(x)=C'_{h-1}(x)C_h(x),\quad h\geq 1.
\end{equation}
The proof of \cite[Theorem~2]{Krattenthaler2001} gives a closed form
$C_h(x)=U_{h}(\frac{1}{2\sqrt{x}})/\big(\sqrt{x}U_{h+1}(\frac{1}{2\sqrt{x}})\big)$
where $U_h(x)$ are the Chebyshev polynomials of the second kind defined by $U_0(x)=1$, $U_1(x)=2x$, and $U_h(x)=2xU_{h-1}(x)-U_{h-2}(x)$ for $h\ge2$. From there, or directly from \eqref{eq:recu-gen}, one proves by induction that
\begin{equation*}
\label{eq:14}
C(\tfrac{1}{4})=2,\quad C_h(\tfrac{1}{4})=2\tfrac{h+1}{h+2},\quad C'_h(\tfrac{1}{4})=2^{h+1}\tfrac{1}{h+2},\quad h\geq 0.
\end{equation*}
Propositions~\ref{prop:levels}, \ref{prop:swift-witness}  together with \ref{prop:volumeDyck} and the above equations now imply
\begin{equation}
\label{eq:sums}
\begin{aligned}
\sum_{i\geq 0} \ell_i(r) &=
\begin{cases}
 \sum_{i\geq 0} \frac{1}{2^{2i+1}}C_i = \tfrac12 C(\tfrac14)=1 & \mbox{if}\ r\notin\QQ, \\
 \sum_{i\geq 0} \frac{1}{2^{2i+1}}C_{i,q-2} = \tfrac12 C_{q-2}(\tfrac14)=\frac{q-1}{q} & \mbox{if}\ r=\tfrac{p}{q}, \\
 \end{cases}\\
\sum_{i\geq 0} sw_i(\tfrac{p}{q}) &= \sum_{i\geq 0}\tfrac{1}{2^{2i+q-1}}C'_{i,q-2}=\tfrac{1}{2^{q-1}}C'_{q-2}(\tfrac{1}{4})=\tfrac{1}{q}.
\end{aligned}
\end{equation}

\begin{theorem}\label{thm:fraction-periodic-points}
The expected fraction of periodic points is $\lim_{n\to\infty}\frac{\expect[\per(X_n,r)]}{n}=\begin{cases}
0 &\mbox{if }r\notin\QQ,\\
\frac{1}{q} &\mbox{if }r=\frac{p}{q}.
\end{cases} $
\end{theorem}

\begin{proof}
For any finite $X\subseteq S^1$ we have $\per(X,r)=|X|-\sum_{i\geq 0}\numlev_i(X,r)$ and also (Section~\ref{SS:qswiftcount}) $\per(X,\frac{p}{q})\geq\sum_{i\geq 0}\swi_i(X,\frac{p}{q})$. Hence the result follows from \eqref{eq:sums} and the bounds 
\[ \lim_{n\to\infty}\frac{1}{n}\expect[\per(X_n,r)]\leq 1-\sum_{i\geq 0}\ell_i(r),\quad \lim_{n\to\infty}\frac{1}{n}\expect[\per(X_n,\tfrac{p}{q})]\geq \sum_{i\geq 0} sw_i(\tfrac{p}{q}). \]
\end{proof}

\begin{theorem}
\label{thm:oneorbit}
If $r=\frac{p}{q}$ then asymptotically almost surely $f_r:X_n\to X_n$ has a single periodic orbit, with winding number $w_n$ and length $\olen_n$ satisfying $\olen_n p-w_nq=1$.
\end{theorem}
\begin{proof}
By Propositions~\ref{prop:swift-witness} and \ref{prop:volumeDyck}, a point of $X_n$ is $q$-swift with probability approaching $\frac{1}{2^{q-1}}C'_{0,q-2}=\frac{1}{2^{q-1}}$. So asymptotically almost surely $X_n$ has a $q$-swift point and hence Lemma~\ref{lem:q-swift} applies.
\end{proof}

\section{The number of periodic points for $r$ irrational}
The results so far give a fairly complete understanding of the asymptotic dynamics of $f_r:X_n\to X_n$ when $r=\frac{p}{q}$. When $r\notin\QQ$ we lack a notion corresponding to $q$-swiftness, and we don't yet have a precise estimate of the number of periodic points other than the $o(n)$ bound of Theorem~\ref{thm:fraction-periodic-points}. In this section we take a first step in that direction.

\begin{definition}
The irrationality exponent of a real number $r$ is the supremum of the set of real numbers $\alpha$ for which the inequality $0<|r-\frac{p}{q}|<\frac{1}{q^\alpha}$ has infinitely many integer solutions $p,q$.
\end{definition}
The irrationality exponent of almost all real numbers (that is, all except a set of Lebesgue measure zero) is $2$ \cite[Theorem E.3]{bugeaud2012distribution}.
 
In Section~\ref{SS:Definitions-and-basic-facts} we defined the winding fraction $\wf(X,r)$ and observed that $\wf(X,r)< r$. One would expect that $\wf(X_n,r)$ approaches $r$ when $n\to\infty$, and the next lemma makes that quantitative.

\begin{lemma}\label{lem:wf-lower-bound}
For $0<r\le1$ we have $\prob[\wf(X_n,r)\geq r-\frac{2\ln n}{n}]\geq 1-\frac{1}{n}$.
\end{lemma}
\begin{proof}
For any $x\in X_n$ the event $\ddist(x,f_r(x))< r-\frac{2\ln n}{n}$ occurs if the arc $[x+r-\frac{2\ln n}{n},x+r)_{S^1}$ contains no point of $X_n$. The length of that arc is $\frac{2\ln n}{n}$, and so this event has probability $(1-\frac{2\ln n}{n})^{n-1}<n^{-2}$ (for $n\ge3$). By the union bound, all $x\in X_n$ satisfy $\ddist(x,f_r(x))\ge r- \frac{2\ln n}{n}$ with probability at least $1-\frac{1}{n}$, and the result follows from \eqref{eq:wf-formula}.
\end{proof}

\begin{theorem}\label{thm:periodic-vertices-lower-bound}
Let $r\in (0,1)$ have irrationality exponent $2$. For any $\varepsilon>0$ the number of periodic points of $f_r$ satisfies $\lim_{n\to\infty}\prob[\per(X_n,r)>n^{1/2-\varepsilon}]=1$, and hence $\expect[\per(X_n,r)]=\Omega(n^{1/2-\varepsilon})$.
\end{theorem}

\begin{proof}
Let $\alpha$ satisfy $\frac{1}{2}-\varepsilon<\frac{1}{\alpha}<\frac{1}{2}$. Since $\alpha>2$, the set of all $p,q$ such that $0<r-\frac{p}{q}<\frac{1}{q^\alpha}$ is finite. Hence for sufficiently large $n$ we have that
$0<r-\frac{p}{q}<\frac{2\ln n}{n}$ implies $\frac{1}{q^\alpha}\le r-\frac{p}{q}$. Let $\wf(X_n,r)=\frac{p}{q}$. By Lemma~\ref{lem:wf-lower-bound}, with probability at least $1-\frac{1}{n}$ we have $0<r-\frac{p}{q}<\frac{2\ln n}{n}$, and therefore $\frac{1}{q^\alpha}\le r-\frac{p}{q}<\frac{2\ln n}{n}$. Hence asymptotically almost surely we have $\per(X_n,r)\geq q > (\frac{n}{2\ln n})^{1/\alpha} >n^{1/2-\varepsilon}$, 
where the last inequality follows since $\frac{1}{2}-\varepsilon<\frac{1}{\alpha}$.
\end{proof}

\section{Application to Vietoris--Rips complexes}\label{S:Application-to-Vietoris-Rips}

For concepts in topology and combinatorial topology we refer to Kozlov \cite{Kozlov}.
The Vietoris--Rips complex \cite{Latschev2001, EdelsbrunnerHarer} captures the topological features of a metric space at a given proximity scale.
\begin{definition}
\label{def:vrcomplexes}
Let $X$ be a metric space and let $r\geq 0$. The \emph{Vietoris--Rips} simplicial complex $\vrc{X}{r}$ has vertex set $X$, and has a finite subset $\sigma\subseteq X$ as a face if and only if the diameter of $\sigma$ is less than $r$.
\end{definition}

For $X\subseteq S^1$ finite, $\vrc{X}{r}$ is the clique complex of the graph connecting each $x\in X$ with points in $X\cap (x-r,x+r)_{S^1}$. If $r>\frac12$ this graph is complete, so here we restrict to the case $0<r<\frac12$.

If $x,y$ are two vertices of a simplicial complex $K$, then we say $x$ is \emph{dominated} by $y$ if $x\in \sigma\in K$ implies $\sigma\cup\{y\}\in K$ for every simplex $\sigma\in K$. If $D\subseteq V(K)$ is a set of vertices, such that each $x\in D$ is dominated by some $y\in V(K)\setminus D$, then $K$ and $K\setminus D$ are homotopy equivalent \cite{BabsonKozlov2006,BarmakStar2013,Matouvsek2008}. The passage from $K$ to $K\setminus D$ is called \emph{dismantling}.

\begin{lemma}
\label{lem:single-dism}
Suppose $X\subseteq S^1$ is finite and $0<r<\frac12$. If $x_0\notin f_r(X)$ then $x_0$ is dominated in $\vrc{X}{r}$ by some vertex in $f_r(X)$.
\end{lemma}
\begin{proof}
Let $x_0'$ be the first point clockwise from $x_0$ which lies in $f_r(X)$. We have $X\cap (x_0-r,x_0'-r]_{S^1}=f_r^{-1}(X\cap [x_0,x_0')_{S^1})=\emptyset$. It follows that $X\cap (x_0-r,x_0+r)_{S^1}\subseteq X\cap (x_0'-r,x_0'+r)_{S^1}$; therefore $x_0$ is dominated by $x_0'$.
\end{proof}

The following is an immediate consequence.
\begin{lemma}\label{lem:dom-level}
Let $X\subseteq S^1$ be finite, let $0<r<\frac12$, and let $i\geq 0$. Then $\vrc{f_r^i(X)}{r}$ dismantles to $\vrc{f_r^{i+1}(X)}{r}$.
\end{lemma}
In particular $\vrc{X}{r}$ dismantles to a homotopy equivalent subcomplex induced by the periodic points of $f_r$. We call this subcomplex the \emph{core} of $\vrc{X}{r}$. It has no dominated vertices, and so it is not dismantlable any further.

\smallskip
We can now state a description of the homotopy type of $\vrc{X}{r}$ exclusively in terms of the dynamical system $f_r$. Below we denote by $\numorb(X,r)$ the number of periodic orbits of $f_r:X\to X$, and $\bigvee^mY$ is a wedge sum of $m$ copies of topological space $Y$. The next proposition is a restatement of Corollary~4.5 of \cite{AA-VRS1} using the notation of this paper.

\begin{proposition}\label{prop:vrchtpy}
If $X\subseteq S^1$ is finite and $0< r<\frac12$, then there is a homotopy equivalence
\begin{equation*}
\vrc{X}{r}\htpyequiv\begin{cases}
S^{2l+1} & \mathrm{if}\ \frac{l}{2l+1}<\wf(X,r)<\frac{l+1}{2l+3}\\
\bigvee^{\numorb(X,r)-1}S^{2l} & \mathrm{if}\ \wf(X,r)=\frac{l}{2l+1}
\end{cases}
\quad\mathrm{for\ some}\ l\in\NN.
\end{equation*}
\end{proposition}
\begin{proof}
By Corollary~4.5 of \cite{AA-VRS1}, we need only consider the case where
\begin{equation}\label{eq:wf-equality}
\frac{l}{2l+1}=\wfc{X}{r}=\frac{k}{n}
\end{equation}
and the core of $f_r\colon X\to X$ is isomorphic to to $\mathrm{Reg}_n^k$. We must show $\numorb(X,r)-1=n-2k-1$, which follows since \eqref{eq:wf-equality} implies
\[ \numorb(X,r)=\gcd(n,k)=\frac{k}{l}=n-2k. \]
\end{proof}

Using these observations we can apply the main results of this paper to deduce the expected asymptotic properties of $\vrc{X_n}{r}$, most notably the expected number of vertices in the core.


\begin{theorem}\label{corollary:fraction-core-points}
Let $X_n\subseteq S^1$ be a sample of $n$ points chosen independently, uniformly at random. Let  $0<r<\frac12$.
\begin{itemize}
\item[(a)] The expected number of vertices in the core of $\vrc{X_n}{r}$ is
\[ \begin{cases}
\frac{1}{q}n+o(n) & \mathrm{if}\ r=\frac{p}{q},\\
o(n) &  \mathrm{if}\ r\notin\QQ.
\end{cases} \]
If $r$ has irrationality exponent $2$, then the expected number of vertices in the core is $\Omega(n^{1/2-\varepsilon})$ for any $\varepsilon>0$.
\item[(b)] Asymptotically almost surely $\vrc{X_n}{r}$ is homotopy equivalent to a sphere of dimension $2\big\lceil\frac{r}{1-2r}\big\rceil-1$.
\end{itemize}
\end{theorem}
\begin{proof}
Part (a) is simply Theorems~\ref{thm:fraction-periodic-points} and \ref{thm:periodic-vertices-lower-bound}. For (b), let $l$ be such that $\frac{l}{2l+1}<r\leq \frac{l+1}{2l+3}$, that is $l=\big\lceil\frac{r}{1-2r}\big\rceil-1$. By Lemma~\ref{lem:wf-lower-bound} a.a.s.\ $\frac{l}{2l+1}<\wf(X_n,r)<r\leq \frac{l+1}{2l+3}$, and (b)  follows from Proposition~\ref{prop:vrchtpy}.
\end{proof}

\bibliographystyle{plain}
\bibliography{RandomCyclicDynamicalSystems}

\appendix

\section{Proofs of asymptotic independence}\label{append:independence}

In this appendix we provide the asymptotic independence results (Lemmas~\ref{lem:K_i}--\ref{lem:S_i(q)}) which are needed for the proofs of Propositions~\ref{prop:levels} and \ref{prop:swift-witness}.

\begin{lemma}\label{lem:DCT}
Fix $m\in\NN$ and let $x=(x_1,\ldots,x_m)\in\RR_+^m$. Let $0<c\le\frac{1}{m}$. Consider the sequence of functions $f_{n,m}\colon\RR_+^m\to\RR$ for $n\ge m+2$ defined by
\[ f_{n,m}(x) = \begin{cases}
\frac{(n-1)(n-2)\cdots(n-m)}{(n-m-1)^m}(1-\frac{\|x\|_1}{n-m-1})^{n-m-1} &\mbox{if }x\in(0,(n-m-1)c)^m, \\
0 &\mbox{otherwise.}
\end{cases} \]
Then for any Lebesgue measurable set $S \subseteq \RR_+^m$ we have
\[ \lim_{n\to\infty}\int_{S} f_{n,m}(x)\ dx=\int_S e^{-\|x\|_1}\ dx. \]
\end{lemma}

\begin{proof}
First we claim $f_{n,m}(x) \to e^{-\|x\|_1}$ pointwise for all $x\in\RR_+^m$. Indeed $\lim_{n\to\infty}(n-m-1)c=\infty$ and
\[ \lim_{n\to\infty}\frac{(n-1)(n-2)\cdots(n-m)}{(n-m-1)^m}=1 \quad\mbox{and}\quad \lim_{n\to\infty}\Bigl(1-\frac{\|x\|_1}{n-m-1}\Bigr)^{n-m-1}=e^{-\|x\|_1}. \]

Next, define $g\colon \RR_+^m\to\RR$ via $g(x)=(m+1)!e^{-\|x\|_1}$. We claim that $|f_{n,m}(x)|\le g(x)$ for all $x\in\RR_+^m$ since the maximum of $\frac{(n-1)(n-2)\cdots(n-m)}{(n-m-1)^m}$ over all $n\ge m+2$ occurs when $n=m+2$ and is equal to $(m+1)!$. Also, for any $x\in(0,(n-m-1)c)^m$ we have
\begin{align*}
\Bigl|\Bigl(1-\frac{\|x\|_1}{n-m-1}\Bigr)^{n-m-1}\Bigr|&=\Bigl(1-\frac{\|x\|_1}{n-m-1}\Bigr)^{n-m-1}\quad\mbox{since }x\in(0,(n-m-1)c)^m\mbox{ with }c\le\tfrac{1}{m} \\
&\le (e^{-\frac{\|x\|_1}{n-m-1}})^{n-m-1}\quad\mbox{since }1-t\le e^{-t}\\
&= e^{-\|x\|_1}.
\end{align*}
Finally, since
\[ \int_{\RR_+^m}g(x)\ dx=(m+1)!\int_0^\infty\cdots\int_0^\infty e^{-x_1-\ldots-x_m}\ dx_m \cdots dx_1=(m+1)!,\]
$g$ is integrable.
Hence the dominated convergence theorem
implies for any Lebesgue measurable set $S \subseteq \RR_+^m$ we have
\[ \lim_{n\to\infty}\int_S f_{n,m}(x)\ dx=\int_S e^{-\|x\|_1}\ dx. \]
\end{proof}

For any set $S \subseteq \RR^m$ and $a\in\RR$ we define $aS = \{ax\ |\ x\in S\}$. A cone $K\subseteq \RR_+^m$ is a set such that if $x\in K$, then $ax\in K$ for all $a\ge0$. Note that if $K$ is a cone and $a>0$ then $aK=K$.

For integers $0<k\le n-2$, define $g_{n,k}\colon\RR_+^k\to\RR$ by 
\[g_{n,k}(u_1,\ldots,u_k)=
\begin{cases}
(n-1)(1-u_1)^{n-2}&\mbox{if }k=1,\\
(n-k)\frac{(1-u_1-\ldots-u_k)^{n-k-1}}{(1-u_1-\ldots-u_{k-1})^{n-k}}&\mbox{otherwise}.
\end{cases}\]

\begin{lemma}\label{lem:joint-pdf}
Suppose $K\subseteq\RR_+^m$ is measurable and a cone. Let $0<c\le\frac{1}{m}$. Suppose $u^{(n)}=u=(u_1,\ldots,u_m)\in\RR_+^m$ is a sequence of random vectors such that 
\begin{itemize}
\item $u\in(0,c)^m$ with probability 1 as $n\to\infty$, and
\item the joint probability density function of $u$ restricted to $K\cap (0,c)^m$ is given by the formula
\[ \prod_{k=1}^m g_{n,k}(u_1,\ldots,u_k)=(n-1)\cdots(n-m)(1-u_1-\ldots-u_m)^{n-m-1}. \]
\end{itemize}
Then
\[ \lim_{n\to\infty}\prob[u\in K]=\int_{K}e^{-\|x\|_1}\ dx. \]
\end{lemma}

\begin{proof}
\begin{align*}
\lim_{n\to\infty}\prob[u\in K] &=\lim_{n\to\infty}\prob[u\in K\cap(0,c)^m] \quad\mbox{since }u\in(0,c)^m\mbox{ with probability 1 as }n\to\infty\\
&=\lim_{n\to\infty}\int_{K\cap(0,c)^m} \prod_{k=1}^m g_{n,k}(u_1,\ldots,u_k)\ du \\
&=\lim_{n\to\infty}\int_{K\cap(0,c)^m} (n-1)\ldots(n-m)(1-u_1-\ldots-u_m)^{n-m-1}\ du \\
&=\lim_{n\to\infty}\int_{(n-m-1)(K\cap(0,c)^m)} f_{n,m}(x)\ dx \quad\mbox{change of variables }x=(n-m-1)u \\
&=\lim_{n\to\infty}\int_{K\cap(0,(n-m-1)c)^m} f_{n,m}(x)\ dx \quad\mbox{since }K\mbox{ is a cone} \\
&=\lim_{n\to\infty}\int_K f_{n,m}(x)\ dx \quad\mbox{since }f_{n,m}\mbox{ vanishes outside }(0,(n-m-1)c)^m \\
&=\int_K e^{-\|x\|_1}\ dx \quad\mbox{by Lemma~\ref{lem:DCT}.}
\end{align*}
\end{proof}

In Lemmas~\ref{lem:K_i} and \ref{lem:K_i(q)} we work in the setting of Section~\ref{sect:prob2}, where $u=(z_1,\ldots,z_{i+1},w_1,\ldots,w_{i+1})$ is a random vector in $R^{2i+2}$. 

\begin{lemma}\label{lem:K_i}
For $r$ irrational we have $\lim_{n\to\infty}\prob[u\in K_i]=\int_{K_i} e^{-\|x\|_1}\ dx$. 
\end{lemma}

\begin{proof}
Let $c'$ be the smallest distance between any two points in $\{x_0, x_0-r, \ldots, x_0-(i+1)r\}$, and let $c=\min\{\frac{c'}{i+1},\frac{1}{2i+2}\}$. As $n\to\infty$, it happens with probability one that $u\in(0,c)^{2i+2}$. Hence for $1\le j\le i+1$ we have $\ddist(x_0-jr,x_j)=w_1+\ldots+w_j<c'$, and for $1\le j\le i$ we have $\ddist(x_0-jr,y_j)=z_1+\ldots+z_{j+1}<c'$. When these equations hold the only pairs of points in $\{x_0,y_0,\ldots,x_i,y_i,x_{i+1}\}$ which are potentially not distinct are $x_j$ and $y_j$ for $1\le j\le i$. We will later condition on the events $H(1,1),\ldots,H(j,j)$, which will imply furthermore that $x_k \neq y_k$ for $1\le k\le j$ and that the arcs along the circle representing the lengths $z_1,w_1,\ldots,z_j,w_j$ are disjoint.

Let $t_1\in(0,c)$. We have $\prob[z_1>t_1]=(1-t_1)^{n-1}$ since there are $n-1$ points (besides $x_0$) that need to miss a region of area $t_1$. So $\prob[z_1<t_1]=1-(1-t_1)^{n-1}$ and the probability density function for $t_1\in(0,c)$ is 
\[f_{z_1}(t_1)=(n-1)(1-t_1)^{n-2}=g_{n,1}(t_1).\]

Let $(t_1,t_2)\in(0,c)^2$. We have $\prob[w_1>t_2|z_1=t_1]=(\frac{1-t_1-t_2}{1-t_1})^{n-2}$ since there are $n-2$ points (besides $x_0$ and $y_0$) in a region of area $1-t_1$ that need to miss a region of area $t_2$. So $\prob[w_1<t_2|z_1=t_1]=1-(\frac{1-t_1-t_2}{1-t_1})^{n-2}$ and the conditional probability density function for $(t_1,t_2)\in(0,c)^2$ is
\[f_{w_1|z_1=t_1}(t_2)=(n-2)\frac{(1-t_1-t_2)^{n-3}}{(1-t_1)^{n-2}}=g_{n,2}(t_1,t_2).\]

Let $(t_1,t_2,t_3)\in(0,c)^3$ and let $t_1>t_2$ (i.e.\ $H(1,1)$ holds). Then the arcs along the circle representing the lengths $z_1,w_1,z_2$ are disjoint. Hence we have $\prob[z_2>t_3|w_1=t_2,z_1=t_1]=(\frac{1-t_1-t_2-t_3}{1-t_1-t_2})^{n-3}$ since there are $n-3$ points in a region of area $1-t_1-t_2$ that need to miss a region of area $t_3$. So the conditional probability density function for $(t_1,t_2,t_3)\in(0,c)^3$ is
\[f_{z_2|z_1=t_1,w_1=t_2,H(1,1)}(t_3)=(n-3)\frac{(1-t_1-t_2-t_3)^{n-4}}{(1-t_1-t_2)^{n-3}}=g_{n,3}(t_1,t_2,t_3).\]

This pattern continues as we condition on the events $H(1,1),\ldots,H(j,j)$. For $(t_1,\ldots,t_{2i+2})\in(0,c)^{2i+2}$ we have the following conditional probability density functions.
\begin{align*}
f_{z_1}(t_1)&=g_{n,1}(t_1) \\
f_{w_1|z_1=t_1}(t_2)&=g_{n,2}(t_1,t_2) \\
f_{z_2|z_1=t_1,w_1=t_2,H(1,1)}(t_3)&=g_{n,3}(t_1,t_2,t_3) \\
f_{w_2|z_1=t_1,w_1=t_2,z_2=t_3,H(1,1)}(t_4)&=g_{n,4}(t_1,t_2,t_3,t_4) \\
&\vdots \\
f_{z_{i+1}|z_1=t_1,\ldots,w_i=t_{2i},H(1,1),\ldots,H(i,i)}(t_{2i+1})&=g_{n,2i+1}(t_1,\ldots,t_{2i+1}) \\
f_{w_{i+1}|z_1=t_1,w_1=t_2,\ldots,z_{i+1}=t_{2i+1},H(1,1),\ldots,H(i,i)}(t_{2i+2})&=g_{n,2i+2}(t_1,\ldots,t_{2i+2}). 
\end{align*}
It follows that the joint probability density function of $u$ on
\[H(1,1)\cap\cdots\cap H(i,i)\cap(0,c)^{2i+2} \supseteq K_i\cap(0,c)^{2i+2}\]
is $\prod_{k=1}^{2i+2} g_{n,k}(u_1,\ldots,u_k)$. Note $K_i\subseteq\RR_+^{2i+2}$ is a measurable cone, and hence Lemma~\ref{lem:joint-pdf} with $m=2i+2$ and $c\le\frac{1}{2i+2}=\frac{1}{m}$ implies $\lim_{n\to\infty}\prob[u\in K_i]=\int_{K_i}e^{-\|x\|_1}\ dx$.
\end{proof}

\begin{lemma}\label{lem:K_i(q)}
For $r=\frac{p}{q}$ rational we have $\lim_{n\to\infty}\prob[u\in K_i(q)]=\int_{K_i(q)} e^{-\|x\|_1}\ dx$. 
\end{lemma}

\begin{proof}

Let $c=\min\{\frac{1}{q(i+1)},\frac{1}{2i+2}\}$. As $n\to\infty$, it happens with probability one that $u\in(0,c)^{2i+2}$. Hence for $1\le j\le i+1$ we have $\ddist(x_0-({j\md q})r,x_j)=w_1+\ldots+w_j<\frac{1}{q}$, and for $1\le j\le i$ we have $\ddist(x_0-({j\md q})r,y_j)=z_1+\ldots+z_{j+1}<\frac{1}{q}$. When these equations hold the only pairs of points in $\{x_0,y_0,\ldots,x_i,y_i,x_{i+1}\}$ which are potentially not distinct are $x_j$ and $y_j$ for $1\le j\le i$, and $y_j$ and $x_{q+j}$ for $0\le j\le i-q+1$. We will condition in order on the events
\[ H(1,1),\ldots,H(q-2,q-2)\quad\mbox{and}\quad H(q-1,q-1),\neg H(1,q-1),\ldots,H(i,i),\neg H(i-q+2,i),\]
which will imply furthermore that the points $\{x_0,y_0,\ldots,x_i,y_i,x_{i+1}\}$ are distinct and that the arcs along the circle representing the lengths $z_1,w_1,\ldots,z_{i+1},w_{i+1}$ are disjoint. Indeed, for $(t_1,\ldots,t_{2i+2})\in(0,c)^{2i+2}$ we have the following conditional probability density functions.
\begin{align*}
f_{z_1}(t_1)&=g_{n,1}(t_1) \\
f_{w_1|z_1=t_1}(t_2)&=g_{n,2}(t_1,t_2) \\
f_{z_2|z_1=t_1,w_1=t_2,H(1,1)}(t_3)&=g_{n,3}(t_1,t_2,t_3) \\
f_{w_2|z_1=t_1,w_1=t_2,z_2=t_3,H(1,1)}(t_4)&=g_{n,4}(t_1,t_2,t_3,t_4) \\
&\vdots \\
f_{z_q|\ldots,H(1,1),\ldots,H(q-1,q-1)}(t_{2q-1})&=g_{n,2q-1}(t_1,\ldots,t_{2q-1}) \\
f_{w_q|\ldots,H(1,1),\ldots,H(q-1,q-1),\neg H(1,q-1)}(t_{2q})&=g_{n,2q}(t_1,\ldots,t_{2q}) \\
f_{z_{q+1}|\ldots,H(1,1),\ldots,H(q-1,q-1),\neg H(1,q-1),H(q,q)}(t_{2q+1})&=g_{n,2q+1}(t_1,\ldots,t_{2q+1}) \\
f_{w_{q+1}|\ldots,H(1,1),\ldots,H(q-1,q-1),\neg H(1,q-1),H(q,q),\neg H(2,q)}(t_{2q+2})&=g_{n,2q+2}(t_1,\ldots,t_{2q+2})\\
&\vdots \\
f_{z_{i+1}|\ldots,H(1,1),\ldots,H(q-1,q-1),\neg H(1,q-1),\ldots,\neg H(i-q+1,i-1),H(i,i)}(t_{2i+1})&=g_{n,2i+1}(t_1,\ldots,t_{2i+1}) \\
f_{w_{i+1}|\ldots,H(1,1),\ldots,H(q-1,q-1),\neg H(1,q-1),\ldots,H(i,i),\neg H(i-q+2,i)}(t_{2i+2})&=g_{n,2i+2}(t_1,\ldots,t_{2i+2}). 
\end{align*}
It follows that the joint probability density function of $u$ on
\[H(1,1)\cap\ldots\cap H(i,i)\cap\neg H(1,q-1)\cap\ldots\cap\neg H(i-q+2,i)\cap(0,c)^{2i+2} \supseteq K_i(q)\cap(0,c)^{2i+2}\]
is $\prod_{k=1}^{2i+2} g_{n,k}(u_1,\ldots,u_k)$. Note $K_i(q)\subseteq\RR_+^{2i+2}$ is a measurable cone, and hence Lemma~\ref{lem:joint-pdf} with $m=2i+2$ and $c\le\frac{1}{2i+2}=\frac{1}{m}$ implies $\lim_{n\to\infty}\prob[u\in K_i(q)]=\int_{K_i(q)}e^{-\|x\|_1}\ dx$.
\end{proof}

In Lemma~\ref{lem:S_i(q)} we let $u=(z_1,\ldots,z_{i+q-2},w_1,\ldots,w_{i+q-2})\in \RR^{2i+2q-2}$, as in Section~\ref{SS:qswiftcount}.

\begin{lemma}\label{lem:S_i(q)}
For $r=\frac{p}{q}$ rational we have $\lim_{n\to\infty}\prob[u\in S_i(q)]=\int_{S_i(q)} e^{-\|x\|_1}\ dx$. 
\end{lemma}

\begin{proof}

Let $c=\min\{\frac{1}{q(i+q-1)},\frac{1}{2i+2q-2}\}$. As $n\to\infty$, it happens with probability one that $u\in(0,c)^{2i+2q-2}$. Hence for $1\le j\le i+q-1$ we have $\ddist(x_0-({j\md q})r,x_j)=w_1+\ldots+w_j<\frac{1}{q}$, and for $1\le j\le i+q-2$ we have $\ddist(x_0-({j\md q})r,y_j)=z_1+\ldots+z_{j+1}<\frac{1}{q}$. When these equations hold the only pairs of points in $\{x_0,y_0,\ldots,x_{i+q-2},y_{i+q-2},x_{i+q-1}\}$ which are potentially not distinct are $x_j$ and $y_j$ for $1\le j\le i+q-2$, and $y_j$ and $x_{q+j}$ for $0\le j\le i-1$. We will condition in order on the events
\[ H(1,1),\ldots,H(q-2,q-2)\quad\mbox{and}\quad H(q-1,q-1),\neg H(1,q-1),\ldots,H(i+q-2,i+q-2),\neg H(i,i+q-2),\]
which will imply furthermore that the points in $\{x_0,y_0,\ldots,x_{i+q-2},y_{i+q-2},x_{i+q-1}\}$ are distinct and that the arcs along the circle representing the lengths $z_1,w_1,\ldots,z_{i+q-1},w_{i+q-1}$ are disjoint. Indeed, for $(t_1,\ldots,t_{2i+2q-2})\in(0,c)^{2i+2}$ we have the following conditional probability density functions.
\begin{align*}
f_{z_1}(t_1)&=g_{n,1}(t_1) \\
f_{w_1|z_1=t_1}(t_2)&=g_{n,2}(t_1,t_2) \\
f_{z_2|z_1=t_1,w_1=t_2,H(1,1)}(t_3)&=g_{n,3}(t_1,t_2,t_3) \\
f_{w_2|z_1=t_1,w_1=t_2,z_2=t_3,H(1,1)}(t_4)&=g_{n,4}(t_1,t_2,t_3,t_4) \\
&\vdots \\
f_{z_q|\ldots,H(1,1),\ldots,H(q-1,q-1)}(t_{2q-1})&=g_{n,2q-1}(t_1,\ldots,t_{2q-1}) \\
f_{w_q|\ldots,H(1,1),\ldots,H(q-1,q-1),\neg H(1,q-1)}(t_{2q})&=g_{n,2q}(t_1,\ldots,t_{2q}) \\
f_{z_{q+1}|\ldots,H(1,1),\ldots,H(q-1,q-1),\neg H(1,q-1),H(q,q)}(t_{2q+1})&=g_{n,2q+1}(t_1,\ldots,t_{2q+1}) \\
f_{w_{q+1}|\ldots,H(1,1),\ldots,H(q-1,q-1),\neg H(1,q-1),H(q,q),\neg H(2,q)}(t_{2q+2})&=g_{n,2q+2}(t_1,\ldots,t_{2q+2})\\
&\vdots \\
f_{z_{i+q-1}|\ldots,H(i+q-2,i+q-2)}(t_{2i+2q-3})&=g_{n,2i+2q-3}(t_1,\ldots,t_{2i+2q-3}) \\
f_{w_{i+q-1}|\ldots,H(i+q-2,i+q-2),\neg H(i,i+q-2)}(t_{2i+2q-2})&=g_{n,2i+2q-2}(t_1,\ldots,t_{2i+2q-2}). 
\end{align*}
It follows that the joint probability density function of $u$ on
\[H(1,1)\cap\ldots\cap H(i+q-2,i+q-2)\cap\neg H(1,q-1)\cap\ldots\cap\neg H(i,i+q-2)\cap(0,c)^{2i+2q-2} \supseteq S_i(q)\cap(0,c)^{2i+2q-2}\]
is $\prod_{k=1}^{2i+2q-2} g_{n,k}(u_1,\ldots,u_k)$. Note $S_i(q)\subseteq\RR_+^{2i+2q-2}$ is a measurable cone, and hence Lemma~\ref{lem:joint-pdf} with $m=2i+2q-2$ and $c\le\frac{1}{2i+2q-2}=\frac{1}{m}$ implies $\lim_{n\to\infty}\prob[u\in S_i(q)]=\int_{S_i(q)}e^{-\|x\|_1}\ dx$.
\end{proof}

\end{document}